\documentclass[11pt]{amsart}
\usepackage{longtable}
\usepackage{array}
\usepackage{multicol}
\usepackage{graphicx}
\usepackage[colorlinks=true,citecolor=black,linkcolor=black,urlcolor=blue]{hyperref}
\makeatletter
 \def\@textbottom{\vskip \z@ \@plus 1pt}
 \let\@texttop\relax
\makeatother
\usepackage{amsmath, amssymb, amsbsy, amsfonts, amsthm, latexsym, amsopn, amstext, amsxtra, euscript, amscd, color, mathrsfs}
\usepackage[normalem]{ulem}
\usepackage{soul}

\usepackage{cite}

\makeatletter

\@namedef{subjclassname@2020}{%
  \textup{2020} Mathematics Subject Classification}
\makeatother

\PassOptionsToPackage{hyphens}{url}\usepackage{hyperref}
 
 \usepackage[capbesideposition=outside,capbesidesep=quad]{floatrow}

\restylefloat{table}
\restylefloat{table}
         
\usepackage{multirow,caption}
            
\usepackage{amscd}
\usepackage{color,enumerate}

\newcommand{\RNum}[1]{\lowercase\expandafter{\romannumeral #1\relax}}

\usepackage[colorinlistoftodos,prependcaption,textsize=tiny]{todonotes}

\theoremstyle{plain}
\newtheorem{thm}{Theorem}[section]
\newtheorem{lem}[thm]{Lemma}

\newtheorem{prop}[thm]{Proposition}

\newtheorem{rmk}[thm]{Remark}

\newtheorem{thm-con}[thm]{Theorem-Conjecture}
\numberwithin{equation}{section}
\newtheorem{def1}[thm]{Definition}

\def\F{{\mathbb F}}
\def\Tr{{\rm Tr}}

\begin{document}
\title[Four new classes of permutation trinomials and their compositional inverses]{Four new classes of permutation trinomials and their compositional inverses}

 \author[S. U. Hasan]{Sartaj Ul Hasan}
 \address{Department of Mathematics, Indian Institute of Technology Jammu, Jammu 181221, India}
  \email{sartaj.hasan@iitjammu.ac.in}
  
  \author[R. Kaur]{Ramandeep Kaur}
  \address{Department of Mathematics, Indian Institute of Technology Jammu, Jammu 181221, India}
  \email{2022rma0027@iitjammu.ac.in}
  \author[H. Kumar]{Hridesh Kumar}
\address{Department of Mathematics, Indian Institute of Technology Jammu, Jammu 181221, India}
\email{2021rma2022@iitjammu.ac.in}
 



\begin{abstract}
We construct four new classes of permutation trinomials over the cubic extension of a finite field with even characteristic. Additionally, we explicitly provide the compositional inverse of each class of permutation trinomials in polynomial form. Furthermore, we derive the compositional inverse of the permutation trinomial \(\alpha X^{q(q^2 - q + 1)} + \beta X^{q^2 - q + 1} + 2X\) for \(\alpha = 1\) and \(\beta = 1\), originally proposed by Xie, Li, Xu, Zeng, and Tang (2023).
\end{abstract}
\maketitle
\section{Introduction}
For a prime power $q=p^m$, let $\F_q$ denote the finite field with $q$ elements, $\F_q^*$  the multiplicative cyclic group of all non-zero elements of $\F_q$, and $\F_q[X]$ the ring of polynomials over $\F_q$ in the indeterminate $X$. A polynomial $f(X)\in \F_q[X]$ is a permutation polynomial (PP) if the induced map $c \mapsto f(c)$ is a bijection from $\F_q$ to itself. The study of permutation polynomials was started by Hermite~\cite{CC} over finite fields of prime order. Later, Dickson~\cite{LE} investigated PPs over arbitrary finite fields. For any PP $f(X)\in \F_q[X]$, there exists a unique  polynomial $f^{-1}(X)\in \F_q[X]$ modulo $X^q-X$ such that $f^{-1}(f(X))\equiv f(f^{-1}(X))\equiv X \pmod{X^q-X}$. The polynomial  $f^{-1}$ is known as the compositional inverse of $f$ over $\F_q$. In 1991, Mullen~\cite{M} proposed the problem of computing the coefficients of the compositional inverses of permutation polynomials. Permutation polynomials and their compositional inverses have a wide range of applications in cryptography\cite{RSA,Schwenk_Cr_98}, coding theory \cite{DH,Chapuy_C_07}, and combinatorial design theory ~\cite{DY}.  For instance, in a block cipher with a Substitution-Permutation Network (SPN) structure,  a permutation is often used as an S-box to build the confusion layer during the encryption process and the compositional  inverse is needed while decrypting the cipher. 

Permutation polynomials with a few terms are of particular interest due to their simple algebraic expressions. The simplest type of polynomials are monomials and monomial $X^d$ permutes $\F_q$ if and only if $\gcd(d, q-1)=1$. In contrast, the permutation behaviour of polynomials with a few terms such as binomials, trinomials, and quadrinomials is still an open problem. Beyond this, providing their compositional inverses is also a challenging and interesting area of research. In this paper, we construct the permutation trinomials of the form $X^d+L(X^s)$ for some positive integers $d$ and $s$, where $L(X)$ is a linearized polynomial over $\F_{q^3}$ and give their compositional inverses as well. 

The investigation of permutation properties of polynomials of the form $X^d+L(X^s)$ was initiated  in \cite{LW, P, PC}. Pasalic \cite{P} has been shown that for any $d$ with $\gcd(d,2^m-1)>1$,  $X^d+L(X)$ can never be a permutation over $\F_{2^m}$ if $L(X) \in \F_2[X]$.   Motivated by this, Li and Wang \cite{LW} proved that for some special $d$ and for any $L(X)\in \F_{2^m}[X]$, $X^d+L(X)$ over $\F_{2^m}$ can never be a permutation polynomial if $\gcd(d, 2^m-1)>1.$
For the Gold exponent \( d = 2^i + 1 \), with \( \gcd(i, m) = 1 \), it has been established that the polynomial 
\(
X^{d} + L(X)
\)
is a permutation of \( \mathbb{F}_{2^m} \) if and only if \( m \) is odd and the linearized part satisfies
$L(X) = a^{2^i} X + a X^{2^i}$
for some \( a \in \mathbb{F}_{2^m}^* \). Furthermore, it was proven, using Hermite’s criterion and properties of Kloosterman sums that $X^{-1}+L(X)$ is not a permutation polynomial over \( \mathbb{F}_{2^m} \) when \( m \geq 5 \)~\cite{LW1}.

In recent years, permutation polynomials with a few terms over the finite field $\F_{q^3}$ of the shape 
\begin{equation}\label{PP}
aX^d+L(X^s)
\end{equation}
 have gained significant attention. In $2020$, Bartoli \cite{B} gave four classes of permutation trinomials of the shape \eqref{PP} over $\F_{q^3}$, where $q=p^m$ and $p>3$. Pang, Xu, Li, and Zeng \cite{PXLZ} proposed six classes of permutation trinomials of the form \eqref{PP} over $\F_{2^{3m}}$. Inspired by this, Gupta, Gahlyan, and Sharma \cite{GGS} investigated two classes of permutation trinomials for even characteristic and one class of permutation trinomials for odd characteristic of the form \eqref{PP}. Moreover, in 2023, Xie, Li, Xu, Zeng, and Tang \cite{XLXZT} constructed two classes of permutation trinomials with the type \eqref{PP} for odd $q$. Recently, two classes  of permutation trinomials of this type for even characteristic are given by Zheng, Kan, Zhang, Peng, and Li  \cite{ZKZPL}.

As alluded to above, finding the compositional inverse of a given permutation polynomial is a challenging and intriguing problem. Recently,  Yuan \cite{Y} proposed a new method, known as the local method, for determining permutation polynomials with their compositional inverses. Using this method, Wu, Yuan, Guan, and Li \cite{WYGL} gave the compositional inverses of permutation trinomials proposed in \cite{GGS}.  For more details on the compositional inverses, we refer the readers to an excellent survey by Wang \cite{W} on the various methods and their unification for computing the compositional inverses of permutation polynomials.

Inspired by the work of several authors, we construct four classes of permutation trinomials of the form $X^d+L(X^s)$ over $\F_{q^3}$, where $q=2^m$.  Moreover, we determine their compositional inverses, and intriguingly, these inverses are expressed explicitly in polynomial form. In \cite [Theorem 1] {XLXZT}, a class of permutation trinomials $\alpha X^{q(q^2-q+1)}+\beta X^{q^2-q+1}+2X$ over $\F_{q^3}$ was constructed. The problem of finding their compositional inverses remains open. Here, we provide the explicit expression for the compositional inverse of the permutation trinomial $\alpha X^{q(q^2-q+1)}+\beta X^{q^2-q+1}+2X$, using the local method, for $\alpha =1$ and $\beta =1$. In addition,  we give a different proof of the class of permutation trinomials $\alpha X^{q(q^2-q+1)}+\beta X^{q^2-q+1}+2X$ for $\alpha =1$ and $\beta =1$. While we present a different proof of a known class of permutation trinomials, however, the interesting fact is that the compositional inverse of this class, which we determine, is a permutation quadrinomial of the form $X^d+L(X^s)$ over $\F_{q^3}$. Somewhat surprisingly, this class of permutation quadrinomials is new in the literature.

 The rest of this paper is organized as follows. In Section \ref{S2}, we introduce some definitions and lemmas, which will be used throughout the paper. Section \ref{S3} presents the construction of permutation trinomials over $\F_{q^3}$ and their compositional inverses. In Section \ref{S4}, we study quasi-multiplicative (QM) equivalence of permutation polynomials proposed in Section \ref{S3} with the known ones and among themselves.

\section{Priliminaries}\label{S2}
In this section, we will discuss some definitions and results that will be used in subsequent sections.
\begin{def1}
A polynomial of the form
$$L(X):=\displaystyle\sum_{i=0}^{n-1}\alpha_{i}X^{q^i}$$
with coefficients in an extension field  $\F_{q^n}$ of $\F_{q}$ is called a $q$-linearized polynomial over $\F_{q^n}$.
\end{def1}

From \cite[Page 362]{LN}, we have that $L(X)$ is a permutation polynomial over $\F_{q^n}$ if and only if the determinant $\det D_L \neq 0$, where
\begin{equation*}
    D_L:=
    \begin{bmatrix}
        \alpha_0 & \alpha_{n-1}^q & \alpha_{n-2}^{q^2} & \cdots & \alpha_{1}^{q^{n-1}}\\
        \alpha_1 & \alpha_{0}^q & \alpha_{n-1}^{q^2} & \cdots & \alpha_{2}^{q^{n-1}}\\
        \alpha_2 & \alpha_{1}^q & \alpha_{0}^{q^2} & \cdots & \alpha_{3}^{q^{n-1}}\\
        \vdots & \vdots & \vdots &  & \vdots\\
        \alpha_{n-1} & \alpha_{n-2}^q & \alpha_{n-3}^{q^2} & \cdots & \alpha_{0}^{q^{n-1}}\\
    \end{bmatrix}.
\end{equation*}
\begin{lem}\label{L21}\cite{W,WL}
Let $L(X)=\displaystyle\sum_{i=0}^{n-1}\alpha_{i}X^{q^i}$ be a $q$-linearized permutation polynomial over $\F_{q^n}$ and $D_{L}$ be its associated Dickson matrix. Then  
$$L^{-1}(X)=(\det D_L)^{-1}\displaystyle\sum_{i=0}^{n-1}\overline \alpha_iX^{q^i},$$ where $\overline \alpha_i$ is the $(i,0)$-th cofactor of the matrix $D_L$.
\end{lem}
\begin{def1}\cite[Page 36]{LN}
    Let $f(X)=a_0X^n+a_1X^{n-1}+\cdots+a_n\in \F_{q}[X]$ and $g(X)=b_0X^m+b_1X^{m-1}+\cdots+b_m\in \F_{q}[X]$ be two polynomials of degree $n$ and $m$ respectively, where $n,m \in \mathbb{N}$. Then the resultant of $f$ and $g$ with respect to $X$ is defined by the determinant 
    \begin{equation*}
        R(f,g,X):=
        \begin{vmatrix}
        a_0 & a_1 & \cdots & a_n & 0 & 0 & \cdots & 0\\
        0 & a_0 & a_1 & \cdots & a_n & 0 & \cdots & 0\\
        \vdots & \vdots & \vdots & \vdots & \vdots & \vdots & \vdots & \vdots\\
        0 & \cdots & 0 & a_0 & a_1 & a_2 & \cdots & a_n\\
         b_0 & b_1 & \cdots & \cdots & b_m & 0 & \cdots & 0\\
         0 & b_0 & b_1 & \cdots& \cdots & b_m  & \cdots & 0\\
         \vdots & \vdots & \vdots & \vdots & \vdots & \vdots & \vdots & \vdots\\
          0 & \cdots & 0 & b_0 & b_1 & b_2 & \cdots & b_m\\
         
        \end{vmatrix}
    \end{equation*}
    of order $m+n$.
\end{def1}
    If $f(X)=a_0(X-\beta_1)(X-\beta_2)\cdots(X-\beta_n)$, where $a_0\neq 0$, in the splitting field of $f$ over $\F_{q}$, then $R(f,g,X)$ is also given by the formula
    \begin{equation*}
        R(f,g,X)=a_{0}^m\displaystyle\sum_{i=1}^{n}g(\beta_i).
    \end{equation*}
 In this case, we obviously have $R(f,g,X)=0$ if and only if $f$ and $g$ have a common divisor in $\F_{q}[X]$ of positive degree.   

 \begin{def1}
 An $n \times n$ circulant matrix is defined as 
 \begin{equation*}
    C:=
    \begin{bmatrix}
       c_0 & c_{n-1} & c_{n-2} & \cdots & c_{1}\\
        c_1 & c_{0} & c_{n-1} & \cdots & c_{2}\\
        c_2 & c_{1} & c_{0} & \cdots & c_{3}\\
        \vdots & \vdots & \vdots &  & \vdots\\
        c_{n-1} & c_{n-2} & c_{n-3}& \cdots & c_{0}\\
    \end{bmatrix}.
\end{equation*}
 \end{def1}
 And, the $\det(C)=\displaystyle\prod_{j=0}^{n-1}(c_0+c_1\omega^j+c_2\omega^{2j}+\cdots+c_{n-1}\omega^{(n-1)j})$, where $\omega$ is a primitive $n$th root of unity.
 \begin{lem}\label{L22}
 \cite{Y} Let $q$ be a prime power and $f(X)$ be a polynomial over $\F_{q}$. Then $f(X)$ is a permutation polynomial over $\F_{q}$ if and only if there exist nonempty finite subsets $S_{i}$, $i=1,2,\ldots,t$ of $\F_q$ and maps $g_i: \F_q\rightarrow S_i$, $i=1,2,\ldots,t$ such that $g_i\circ f=h_i$, $i=1,2,\ldots,t$ and $X=F(h_1(X),h_2(X),\ldots, h_t(X))$ for all $X\in \F_q$, where $F(X_1,X_2,\ldots, X_t)\in \F_q[X_1,X_2,\ldots,X_t]$. Moreover, the compositional inverse of $f$ is given by $f^{-1}(X)=F(g_1(X),g_2(X),\ldots,g_t(X))$.
 \end{lem}
 \section{Permutation trinomials}\label{S3}
 This section aims to study the permutation trinomials of the form $X^d+L(X^s)$ over $\F_{q^3}$. More precisely, we construct four classes of permutation trinomials over $\F_{q^3}$, where $q=2^m$ and $m$ is a positive integer. We also provide the explicit form of the compositional inverse of each class of permutation trinomials. Furthermore, we give the compositional inverse of permutation trinomial $\alpha X^{q(q^2-q+1)}+\beta X^{q^2-q+1}+2X$ proposed in \cite [Theorem 1] {XLXZT} for $\alpha =1$ and $\beta =1$.
 \begin{thm}\label{T31}
 Let $q=2^m$, where $m$ is a positive integer. Then $f_1(X)=X^{2q^2}+aX^{q^2+q-1}+bX^{q^2-q+1}$, for $a,b\in \F_{q}^*$ is a permutation polynomial over $\F_{q^3}$ if and only if $b=a$.
 \end{thm}
 \begin{proof} We first suppose that $a=b$. To prove $f_1(X)$ is a permutation polynomial over $\F_{q^3}$, it suffices to show that for any $A\in \F_{q^3}$, the equation
 \begin{equation}\label{me31}
 X^{2q^2}+aX^{q^2+q-1}+aX^{q^2-q+1}=A
 \end{equation}
  has at most one solution in $\F_{q^3}$. We can easily see that $X=0$ is a  solution of Equation \eqref{me31} for $A=0$ and $X=0$ is not a solution of Equation \eqref{me31} when $A\neq 0$. Therefore, it is enough to consider $X\in \F_{q^3}^*$.   Let $Y=X^q, Z=X^{q^2}, B=A^q$ and  $C=A^{q^2}$, then from Equation \eqref{me31}, we obtain the following system of equations
\begin{equation}\label{e31}
Z^2+a\frac{ZX}{Y}+a\frac{ZY}{X}=A,
\end{equation}
\begin{equation}\label{e32}
X^2+a\frac{XY}{Z}+a\frac{XZ}{Y}=B,
\end{equation}
\begin{equation}\label{e33}
Y^2+a\frac{YZ}{X}+a\frac{YX}{Z}=C.
\end{equation}
By adding the above three equations, we get $X^2+Y^2+Z^2=A+B+C$. To show that  Equation \eqref{me31} has a unique solution in $\F_{q^3}$, we consider the following two cases.

\textbf{Case 1.} If $A \in \F_q$,  then we have $X^2+Y^2+Z^2=A$. Using this in Equations \eqref{e31}, \eqref{e32}, and \eqref{e33}, we obtain
\begin{equation*}
(X^2+Y^2)(1+\frac{aZ}{XY})=0,
\end{equation*}
\begin{equation*}
(Y^2+Z^2)(1+\frac{aX}{ZY})=0,
\end{equation*}
\begin{equation*}
(X^2+Z^2)(1+\frac{aY}{XZ})=0.
\end{equation*}
The above equations imply that  $X=Y=Z$. We put $X=Y=Z$ in Equation \eqref{e31}, which gives that $X=A^{\frac{1}{2}}$. If $A=0$,  then $X=0$, which is a contradiction. Thus, Equation \eqref{me31} has a unique solution in $\F_{q^3}$ for each $A \in \F_q$.

\textbf{Case 2.} Let $A\in \F_{q^3}\setminus{\F_{q}}$.  We use  $X^2+Y^2+Z^2=A+B+C$  in Equations \eqref{e31}, \eqref{e32} and \eqref{e33} to get
\begin{equation*}
\begin{cases}
&a(X^2+Y^2)Z=(B+C+X^2+Y^2)XY,\\
 &Y(B+X^2)Z=aX(A+B+C+X^2), \\
 &X(C+Y^2)Z=aY(A+B+C+Y^2).
 \end{cases}
 \end{equation*}
 If $X^2= Y^2$ or $X^2= B$ or $Y^2 = C$, then $B=C$ or $A=C$ or $A=B$. This gives that $A \in \F_q$, which is not possible. Therefore, we obtain
\begin{equation*}
Z=\frac{(B+C+X^2+Y^2)XY}{a(X^2+Y^2)}=\frac{aX(A+B+C+X^2)}{Y(B+X^2)}=\frac{aY(A+B+C+Y^2)}{X(C+Y^2)},
\end{equation*}
which further implies 
\begin{equation*}
 \begin{cases}
 &\dfrac{(B+C+X^2+Y^2)XY}{a(X^2+Y^2)}=\dfrac{aX(A+B+C+X^2)}{Y(B+X^2)} , \\
& \dfrac{(B+C+X^2+Y^2)XY}{a(X^2+Y^2)}=\dfrac{aY(A+B+C+Y^2)}{X(C+Y^2)}.
\end{cases}
\end{equation*}
The above identities reduce to the following quartic equations in $Y$
\begin{equation*}
\begin{split}
g(Y):=&(X^2+B)Y^4+(X^4+CX^2+a^2X^2+B^2+BC+Aa^2+Ba^2+Ca^2)Y^2+a^2X^4+Aa^2X^2\\&+Ba^2X^2+Ca^2X^2=0 ,
\end{split}
\end{equation*}
and
\begin{equation*}
\begin{split}
h(Y):=&(X^2+a^2)Y^4+(X^4+BX^2+a^2X^2+Aa^2+Ba^2+Ca^2)Y^2+CX^4+BCX^2+C^2X^2\\&+Aa^2X^2+Ba^2X^2+Ca^2X^2=0.
\end{split}
\end{equation*}
The resultant $R(g,h,Y)$ of $g(Y)$ and $h(Y)$ is given by
\[
R(g,h,Y):=a^4(B+C)^4p(X)^2,
\]
where 
\[
p(X)=(AC+Aa^2+Ca^2+a^4)X^4+A^2Ba^2+A^2a^4+AB^2C+AB^2a^2+B^2Ca^2+B^2a^4+BC^2a^2+C^2a^4.
\]
As $g(Y)$ and $h(Y)$ have a common solution, so $R(g,h,Y)=0$. Since $A\in \F_{q^3}\setminus{\F_{q}}$, it follows that $B\neq C$. Thus, we must have $p(X)=0$, i.e.,  
\[
(AC+Aa^2+Ca^2+a^4)X^4+A^2Ba^2+A^2a^4+AB^2C+AB^2a^2+B^2Ca^2+B^2a^4+BC^2a^2+C^2a^4=0.
\]
Notice that  $AC+Aa^2+Ca^2+a^4 \neq 0$. If $AC+Aa^2+Ca^2+a^4=0$, then $A=a^2$ or $C=a^2$, which turns out that $A\in \F_{q}$. This is not the case. Therefore, we obtain 
\[
X=\left(\frac{A^2Ba^2+A^2a^4+AB^2C+AB^2a^2+B^2Ca^2+B^2a^4+BC^2a^2+C^2a^4}{AC+Aa^2+Ca^2+a^4}\right)^{\frac{1}{4}}.
\]
Hence, Equation \eqref{me31} has a unique solution in $\F_{q^3}$ for each $A\in \F_{q^3}\setminus \F_{q}$ as well. Consequently, when $a=b$, $f_1(X)$ is a permutation polynomial over $\F_{q^3}$.  

Conversely, suppose that $a \neq b$.  We will show that $f_1(X)$ is not a permutation polynomial. Since $0 \neq a+b\in \F_{q}$, 
\[
f_1(a+b)=(a^2+b^2)+a(a+b)+b(a+b)=a^2+b^2+a^2+ab+ba+b^2=0.
\] 
Therefore, we get an element $a+b\neq 0$ such that $f_1(a+b)=0=f_1(0)$. Thus, $f_1(X)$ is not a permutation polynomial over $\F_{q^3}$. This completes the proof. 
 \end{proof}
 In the following proposition, we provide the compositional inverse of the permutation trinomial $f_1(X)$. 
 \begin{prop}\label{P31}
 The compositional inverse of the permutation trinomial obtained in Theorem \ref{T31} is given by 
 \[
 f_1^{-1}(X)=(X+a^2)^{q^3-1}P_3(X)+a\prod_{\beta \neq a^2}(X-\beta)^{q^3-1},
 \]
 where \[P_3(X)^4=(a^2P_1(X)+a^4\Tr_{q}^{q^3}(X)^2+X^{1+2q+q^2})P_2(X)^{q^3-q^2-2},\]
 and $P_1(X)=X^{q+2}+X^{2q+1}+X^{q^2+2q}+X^{2q^2+q}$, $P_2(X)=(X+a^2)$, $\Tr_{q}^{q^3}(X)=X+X^q+X^{q^2}$.
 \end{prop}
 \begin{proof}
 For $X\in \F_{q^3}$, let $g_1(X)=X, g_2(X)=X^q, g_3(X)=X^{q^2}, h_1(X)=g_1\circ f_1(X)=f_1(X), h_2(X)=g_2\circ f_1(X)=f_1^{q}(X)$, and $h_3(X)=g_3\circ f_1(X)=f_1^{q^2}(X)$. From the proof of  Theorem \ref{T31}, we have
  \begin{equation*}
X=\begin{cases}
 A^{\frac{1}{2}} & \text{ if } A \in \F_q,\\
\left(\frac{A^2Ba^2+A^2a^4+AB^2C+AB^2a^2+B^2Ca^2+B^2a^4+BC^2a^2+C^2a^4}{AC+Aa^2+Ca^2+a^4}\right)^{\frac{1}{4}} & \text { if } A \in \F_{q^3}\setminus \F_q,\\
\end{cases}
\end{equation*}
where $A=h_1(X), B=A^q=h_2(X)$, and $C=A^{q^2}=h_3(X)$.
Thus, in view of Lemma \ref{L22},
\begin{equation*}
F(X_1,X_2,X_3)=\begin{cases}X_1^{\frac{1}{2}}& \text{ if } X_1 \in \F_q,\\
\left(\frac{X_1^2X_2a^2+X_1^2a^4+X_1X_2^{2}X_3+X_1X_2^{2}a^2+X_2^{2}X_3a^2+X_2^{2}a^4+X_2X_3^{2}a^2+X_3^{2}a^4}{X_1X_2+X_1a^2+X_3a^2+a^4}\right)^{\frac{1}{4}} & \text{ if } X_1 \in \F_{q^3}\setminus \F_q
\end{cases}
\end{equation*}
 and the compositional inverse of $f_1(X)$ is given by
\begin{equation*}
\begin{split}
f_1^{-1}(X)&= F(X, X^q, X^{q^2})\\&=
 \begin{cases}
X^{\frac{q^3}{2}}, \text { if } X\in \F_{q},\\
 \left(\frac{a^2X^{2+q}+a^4X^2+X^{1+2q+q^2}+a^2X^{1+2q}+a^2X^{2q+q^2}+a^4X^{2q}+a^2X^{q+2q^2}+a^4X^{2q^2}}{X^{1+q^2}+a^2X+a^2X^{q^2}+a^4}\right)^{\frac{q^3}{4}}, \text{ if } X\in \F_{q^3}\setminus{\F_{q}}
 \end{cases}
 \\&=\begin{cases}
X^{\frac{q^3}{2}}, \text { if } X = a^2,\\
 \left(\frac{a^2(X^{2+q}+X^{1+2q}+X^{2q+q^2}+X^{q+2q^2})+a^4(X+X^{q}+X^{q^2})^2+X^{1+2q+q^2}}{(X+a^2)^{q^2+1}}\right)^{\frac{q^3}{4}}, \text{ if } X\neq a^2
 \end{cases}
 \\&=\begin{cases}
X^{\frac{q^3}{2}}, \text { if } X = a^2,\\
 \left((a^2(P_1(X))+a^4\Tr_{q}^{q^3}(X)^2+X^{1+2q+q^2})P_2(X)^{q^3-q^2-2}\right)^{\frac{q^3}{4}},  \text{ if } X\neq a^2.
 \end{cases}
 \end{split}
 \end{equation*}
 From  the last expression of $f_1^{-1}(X)$, we can write $f_1^{-1}(X)$ as follows
  \[
 f_1^{-1}(X)=(X+a^2)^{q^3-1}P_3(X)+a\prod_{\beta \neq a^2}(X-\beta)^{q^3-1}.
 \]

 \end{proof}
\begin{thm}\label{T32}
 Let $q=2^m$, where $m$ is a positive integer. Then $f_2(X)=aX^{1+q}+bX^{1+q^2}+X^{2q^2+2q}$, where $a,b\in \F_{q}^*$, is a permutation polynomial over $\F_{q^3}$ if and only if $b=a$.
\end{thm}
\begin{proof}
First we will show that $f_2(X)$ is a permutation polynomial over $\F_{q^3}$ for $a=b$. To do this, it is enough to prove that for $a=b$,  the equation
 \begin{equation}\label{me32}
 f_2(X+A)=f_2(X)
 \end{equation}
  has no solution in $\F_{q^3}$ for any  $A\in \F_{q^3}^*$.  If possible, we suppose that $f_2(X+A)=f_2(X)$ has a solution in $\F_{q^3}$ for some $A\in \F_{q^3}^*$.  Let $Y=X^q$, $Z=Y^q$, $B=A^q$, and $C=B^q$.  Then Equation \eqref{me32} renders the following system of equations
\begin{equation*}
\begin{cases}
a(X+A)(Y+B)+a(X+A)(Z+C)+(Z^2+C^2)(Y^2+B^2)=aXY+aXZ+Y^2Z^2,\\
a(Y+B)(Z+C)+a(Y+B)(X+A)+(X^2+A^2)(Z^2+C^2)=aYZ+aXY+X^2Z^2,\\
a(Z+C)(X+A)+a(Z+C)(Y+B)+(Y^2+B^2)(X^2+A^2)=aXZ+aYZ+^2Y^2.
\end{cases}
\end{equation*}
The above system can be reduced to 
\begin{equation}\label{e34}
aBX+aAY+aAB+aXC+aAZ+aAC+B^2Z^2+C^2Y^2+B^2C^2=0,
\end{equation}
\begin{equation}\label{e35}
aCY+aBZ+aBC+aAY+aBX+aAB+C^2X^2+A^2Z^2+A^2C^2=0,
\end{equation}
\begin{equation}\label{e36}
aAZ+aCX+aAC+aBZ+aCY+aBC+A^2Y^2+B^2X^2+A^2B^2=0.
\end{equation}
We now add the above three equations to get 
\begin{equation*}
B^2Z^2+C^2Y^2+B^2C^2+C^2X^2+A^2Z^2+A^2C^2+A^2Y^2+B^2X^2+A^2B^2=0,
\end{equation*}
which is equivalent to
\begin{equation}\label{e37}
BZ+CY+BC+CX+AZ+AC+AY+BX+AB=0.
\end{equation}
Next, by using Equation \eqref{e37} in Equation \eqref{e34}, Equation \eqref{e35}, and Equation \eqref{e36}, we obtain 
\begin{equation*}
\begin{cases}
(BZ+CY+BC)(a+BZ+CY+BC)=0,\\
(CX+AZ+AC)(a+CX+AZ+AC)=0,\\
(AY+BX+AB)(a+AY+BX+AB)=0.
\end{cases}
\end{equation*}
This implies  
$$BZ+CY+BC=CX+AZ+AC=AY+BX+AB=0$$
 or 
 $$BZ+CY+AZ+AC=CX+AZ+AC=AY+BX+AB=a.$$
We first suppose that $BZ+CY+BC=CX+AZ+AC=AY+BX+AB=0$. Then using $BZ+CY+BC=CX+AZ+AC=0$, we get 
$$Z=\frac{BC+CY}{B}=\frac{AC+CX}{A}.$$ The last equality implies that $AY=BX$, which gives that $Y=\frac{B}{A}X$, $Z=\frac{C}{B}Y$ and $Z=\frac{C}{A}X$. By substituting  $Z=\frac{C}{A}X$ in $Z=\frac{AC+CX}{A}$, we get $C=0$. This is not possible. Therefore, we must have  $$BZ+CY+BC=CX+AZ+AC=AY+BX+AB=a,$$ and using these values in Equation \eqref{e37}, we obtain that $a=0$, which is a contradiction. Hence, Equation \eqref{me32} has no solution in $\F_{q^3}$ for any $A\in\F_{q^3}^*$.

Conversely, we assume that $a\neq b$. As $ a^{1/2}+b^{1/2} \in \F_q$, $f_2(0)=f_2(a^{1/2}+b^{1/2})=0$. Since $a^{1/2}+b^{1/2} \neq 0$,  $f_2(X)$ is not a permutation polynomial over $\F_{q^3}$. 
\end{proof}
In the following proposition, we give the compositional inverse of the permutation trinomial $f_2(X)$. Observe that for $a=b$,  $f_2(X)=L(X^{q+1})$, where $L(X)=aX+X^{2q}+aX^{q^2}=aX+X^{2^{m+1}}+aX^{2^{2m}}$. The associated Dickson matrix $D_{L}$ of $2$-linearized polynomial  $L(X)$ over $\F_{q^3}$ is given by \[
D_{L}:=\begin{bmatrix}
  M_1 & M_2 & M_3 \\
  \end{bmatrix},
  \]
  where
   \begin{equation*} 
   \begin{split} 
   &M_1:=(a_{ij})_{3m \times (m)} \text{ with }
  a_{ij}=\begin{cases} a^{2^{j-1}} & \text{ if } i=j, 2m+j,\\
  1 & \text{ if } i=m+j+1,\\
  0, & \text{otherwise},
  \end{cases}\\
  &M_2:=(b_{ij})_{3m \times (m-1)} \text{ with }
  b_{ij}=\begin{cases} a^{2^{m+j-1}} & \text{ if } i=j,m+j,\\
  1 & \text{ if } i=2m+j+1,\\
  0, & \text{otherwise},
  \end{cases}\\
   \end{split}
  \end{equation*}
  and
     \begin{equation*} 
   \begin{split}
  M_3:=(c_{ij})_{3m \times (m+1)} \text{ with }
  c_{ij}=\begin{cases} a^{2m+j-2} & \text{ if } i=m+j-1,2m+j-1,\\
   1 & \text{ if } i=j,\\
  0, & \text{otherwise}.
  \end{cases}
  \end{split}
  \end{equation*}

\begin{prop}\label{P32} The compositional inverse of the permutation trinomial given in Theorem \ref{T32} is 
\[
f_2^{-1}(X)=({L^{-1}(X)})^{\frac{q^2-q+1}{2}},
\]
where $L^{-1}(X)=(\det D_L)^{-1}\displaystyle\sum_{i=0}^{3m-1}\overline \alpha_iX^{q^i}$, $\overline \alpha_i$ is the $(i,0)$-th cofactor of the matrix $D_L$, which is defined as above. \end{prop}
\begin{proof}
Since $f_2(X)=L(X^{q+1})$, $f_2^{-1}(X)=({L^{-1}(X)})^{\frac{q^2-q+1}{2}}$, where $(\frac{q^2-q+1}{2})(q+1)\equiv 1 \pmod {q^3-1}$. Moreover, from  Lemma \ref {L21}, $$L^{-1}(X)=(\det D_L)^{-1}\displaystyle\sum_{i=0}^{3m-1}\overline \alpha_iX^{q^i},$$ where $\overline \alpha_i$ is the $(i,0)$-th cofactor of the matrix $D_L.$

\end{proof}
\begin{thm}\label{T33}
Let $q=2^m$, where $m$ is a positive integer. Then $f_3(X)=X^{1+q}+X^{1+q^2}+X^{2q+2}$ is a permutation polynomial over $\F_{q^3}$ if and only if $m\not \equiv 1 \mod 3$.
\end{thm}
\begin{proof}
Note that $f_3(X)=X^{1+q}+X^{1+q^2}+X^{2q+2}=X^{q^3+q}+X^{1+q^2}+X^{2q+2}=(X+X^2+X^{q^2})\circ (X^{1+q})$. To prove the result, it suffices to show that $X+X^2+X^{q^2}$ and $X^{1+q}$ are permutation polynomials over $\F_{q^3}$ if and only if $m\not\equiv 1 \mod 3$. First we will show that $\gcd(1+q, q^3-1)=1$, which implies that monomial $X^{1+q}$ is always a PP over $\F_{q^3}$.  On the contrary, we suppose that $\gcd(1+q, q^3-1)=d \neq 1$. This implies that $p_1\mid 1+q$ and $p_1\mid q^3-1$, where $p_1$ is a prime divisor of $d$. Thus, we have that $p_1\mid q+1$, $p_1\mid q-1$ or $p_1\mid q+1$, $p_1\mid q^2+q+1$, and the consequences would imply that $p_1\mid 2$ or $p_1\mid q^2$, which is a contradiction as $p_1\mid 1+q$ and $q$ is even.  Next, we will show that the $2$-linearized polynomial $L_1(X):=X+X^2+X^{q^2}$ is a PP over $\F_{q^3}$ if and only if $m\not \equiv 1 \mod 3$. The Dickson matrix associated to the $2$-linearized polynomial $L_1(X)$  is given by 
\[
D_{L_1}:=\begin{bmatrix}
  N_1 & N_2 & N_3 \\
  \end{bmatrix},
  \]
  where
   \begin{equation*} 
   \begin{split} 
   &N_1:=(d_{ij})_{3m \times (m+1)} \text{ with }
  d_{ij}=\begin{cases} 1 & \text{ if } i=j,j+1,2m+j-1,\\
  0, & \text{otherwise},
  \end{cases}\\
  &N_2:=(e_{ij})_{3m \times (2m-2)} \text{ with }
  e_{ij}=\begin{cases} 1 & \text{ if } i=j,m+j+1,m+j+2,\\
  0, & \text{otherwise},
  \end{cases}\\
   \end{split}
  \end{equation*}
  and
     \begin{equation*} 
   \begin{split}
  N_3:=(f_{ij})_{3m \times 1} \text{ with }
  f_{i1}=\begin{cases} 1 & \text{ if } i=1,2m-1,3m,\\
  0, & \text{otherwise}.
  \end{cases}
  \end{split}
  \end{equation*}
  
Since $D_{L_1}$ is a circulant matrix, $\det D_{L_1}=\displaystyle \prod_{j=0}^{n-1}(1+\omega^j+(\omega^{j})^{2m})$, where $n=3m$ and $\omega$ is a primitive $n$th root of unity. We know that $L_1(X)$ is a  PP over $\F_{q^3}$ if and only if $\det D_{L_1}\neq 0$. Therefore, we will prove that $\det D_{L_1}\neq 0$ if and only if $m\not \equiv 1 \mod 3$. First we assume that $m \equiv 1 \mod 3$. For $j=m <3m$, we have $1+\omega^j+(\omega^j)^{2m}=1+\omega^m+(\omega^m)^{2m}$. Since $m \equiv 1 \mod 3$, $1+\omega^m+(\omega^m)^{2m}=1+\omega^m+\omega^{2m}$  and $1+\omega^m+\omega^{2m}=0$ as $\omega$ is a primitive $3m$th root of unity. Hence, $\det D_{L_1}=0$. 

Next, we suppose that $m\not \equiv 1 \mod 3$. Our aim is to show that $\det D_{L_1} \neq 0$. On the contrary, we assume that $\det D_{L_1}=0$, i.e., there exists a $0 \leq j \leq n-1$ such that $1+\omega^j+(\omega^j)^{2m}=0$. To this end, we consider the following three cases.

\textbf{Case 1.} Let $j \equiv 0 \mod 3$. Then $1+\omega^j+(\omega^j)^{2m}=2+\omega^j=0$, which implies  that $\omega=0$. This is not possible.

\textbf{Case 2.} In this case, we consider $j \equiv 1 \mod 3$. Thus,  $1+\omega^j+(\omega^j)^{2m}=1+\omega^j+\omega^{2m}$. Therefore, 
\begin{equation}\label{e38}
1+\omega^j+\omega^{2m}=0.
\end{equation}
Now, by multiplying $\omega^{m}$ with Equation \eqref{e38}, we get 
\begin{equation}\label{e39}
\omega^m+\omega^{j+m}+1=0
\end{equation}
as $\omega^{3m}=1$.  We add  Equation \eqref{e38} and Equation \eqref{e39} so as to get $\omega^j+\omega^{j+m}+\omega^m+\omega^{2m}=0$. This gives that  $\omega^m=1$ or $\omega^j=\omega^m$. Since $\omega$ is a primitive $3m$th root of unity, $\omega^m\neq 1$ and, then we must have $\omega^j=\omega^m$. It gives that $3m \mid \pm(j-m)$.  Note that $m \not\equiv 1 \mod 3$ and $j \equiv 1 \mod 3$, so $j\neq m$. Thus, it contradicts the condition $3m \mid \pm(j-m)$ as $0 \leq j,m \leq 3m-1$. 

\textbf{Case 3.} If $j \equiv 2 \mod 3$.  Then $1+\omega^j+(\omega^j)^{2m}=1+\omega^j+\omega^m$.  Thus,
\begin{equation}\label{e310}
1+\omega^j+\omega^{m}=0.
\end{equation}
We mutiply $\omega^{2m}$ with Equation \eqref{e310} to get 
\begin{equation}\label{e311}
\omega^{2m}+\omega^{j+2m}+1=0.
\end{equation}
By adding  Equation \eqref{e310} and Equation \eqref{e311}, we obtain $\omega^j+\omega^m+\omega^{2m}+\omega^{j+2m}=0$. This implies that $\omega^m=1$ or $\omega^j+\omega^{j+m}+\omega^m=0$. Since $\omega$ is a primitive $3m$th root of unity, $\omega^j+\omega^{j+m}+\omega^{m}=0.$ By using  $\omega^j+\omega^{j+m}+\omega^m=0$ in Equation \eqref{e310}, we get $\omega^{j+m}=1$, which gives $3m \mid j+m$. Since $m \not \equiv 1 \mod 3$, so  $m \equiv 0 \mod 3$ or $m\equiv 2 \mod 3$. If $m \equiv 0 \mod 3$, then the condition $3\mid 3m \mid j+m$ implies that $3\mid j$, which is a contradiction as $j \equiv 2\mod 3$. Next, if $m \equiv 2 \mod 3$, then $3\mid 3m \mid j+m$ gives that $3\mid j+2$, which is again not possible. Thus, $\det D_{L_1} \neq 0$. Hence, $f_3(X)$ is a PP over $\F_{q^3}$ if and only if $ m \not \equiv 1 \mod 3$.
\end{proof}
\begin{thm}\label{T34}
Let $q=2^m$, where $m$ is a positive integer. Then $f_4(X)=X^{1+q}+X^{1+q^2}+X^{2q^2+2}$ is a permutation polynomial over $\F_{q^3}$ if and only if $m\not \equiv 2 \mod 3$.
\end{thm}
\begin{proof}
We observe that, over $\F_{q^3}$, $f_4(X)=X^{1+q}+X^{1+q^2}+X^{2q^2+2}=X^{q^3+q}+X^{1+q^2}+X^{2q^2+2}=(X+X^2+X^q)\circ X^{1+q^2}$. Therefore, to check the permutation behaviour of $f_4(X)$, it suffices to check the permutation behaviour of the $2$-linearized polynomial $L_2(X):=X+X^2+X^q$ and monomial $X^{1+q^2}$. Clearly, $\gcd(1+q^2, q^3-1)=1$, which shows that $X^{1+q^2}$ is a permutation polynomial over $\F_{q^3}$.   It remains to show that the $2$-linearized polynomial $L_2(X)$ is a permutation polynomial over $\F_{q^3}$ if and only if $m\not \equiv 2 \mod 3$. The Dickson matrix $D_{L_2}$ associated with the $2$-linearized polynomial $L_2(X)$ is defined as follows
  
\[
D_{L_2}:=\begin{bmatrix}
  N'_1 & N'_2 & N'_3 \\
  \end{bmatrix},
  \]
  where
   \begin{equation*} 
   \begin{split} 
   &N'_1:=(d'_{ij})_{3m \times (2m+1)} \text{ with }
  d'_{ij}=\begin{cases} 1 & \text{ if } i=j,j+1,m+j-1,\\
  0, & \text{otherwise},
  \end{cases}\\
  &N'_2:=(e'_{ij})_{3m \times (m-2)} \text{ with }
  e'_{ij}=\begin{cases} 1 & \text{ if } i=j,2m+j+1,2m+j+2,\\
  0, & \text{otherwise},
  \end{cases}\\
   \end{split}
  \end{equation*}
  and
     \begin{equation*} 
   \begin{split}
  N'_3:=(f'_{ij})_{3m \times 1} \text{ with }
  f'_{i1}=\begin{cases} 1 & \text{ if } i=1,m-1,3m,\\
  0, & \text{otherwise}.
  \end{cases}
  \end{split}
  \end{equation*} 
Furthermore, $D_{L_2}$ is a circulant matrix, thus, $\det D_{L_2} =\displaystyle \prod_{j=0}^{n-1}(1+\omega^j+(\omega^{j})^{m})$, where $n=3m$ and $\omega$ is a primitive $n$th root of unity. We know that $L_2(X)$ is a PP over $\F_{q^3}$ if and only if $\det D_{L_2} \neq 0$. Therefore, we will prove that $\det D_{L_2} \neq 0$ if and only if $m\not \equiv 2 \mod 3$. First we assume that $m \equiv 2 \mod 3$. If we take $j=2m<3m$, then $1+\omega^{2m}+(\omega^{m})^{2m}=1+\omega^{2m}+\omega^{m}$ as $m \equiv 2\mod 3.$ Since $\omega$ is a primitive $3m$th root of unity, $1+\omega^{2m}+\omega^{m}=0$. It gives that $\det D_{L_2}=0$. 

Next, we suppose that $m\not\equiv 2\mod3.$ Our aim is to show that $\det D_{L_2} \neq 0.$ If possible, we assume that $\det D_{L_2}=0$, then there exists a $0<j<n$ such that $1+\omega^{j}+(\omega^{j})^{2m}=0.$ By using similar arguments as in the proof of Theorem \ref{T33}, we will get a contradiction to $\det D_{L_2}=0.$ Thus, $f_4(X)$ is a permutation polynomial over $\F_{q^3}$ if and only if $m\not \equiv 2 \mod 3.$
\end{proof}
\begin{prop}\label{P33}
Let $f_3(X)$ and $f_4(X)$ be two permutation trinomials obtained in Theorem \ref{T33} and  Theorem \ref{T34}, respectively. Then 
\begin{equation*}
\begin{split}
f_3^{-1}(X)=(L_1^{-1}(X))^{\frac{q^2-q+1}{2}} 
\end{split}
\end{equation*}
and 
\begin{equation*}
\begin{split}
f_4^{-1}(X)=(L_2^{-1}(X))^{\frac{q^3-q^2+q}{2}}, 
\end{split}
\end{equation*}
where $L_1^{-1}(X)=(\det D_{L_1})^{-1}\displaystyle\sum_{i=0}^{3m-1}\overline \alpha_i X^{q^i}$, $L_2^{-1}(X)=(\det D_{L_2})^{-1}\displaystyle\sum_{i=0}^{3m-1}\overline \beta_i X^{q^i}$, $\overline \alpha_i$ is the $(i,0)$-th cofactor of $D_{L_1}$, and $\overline \beta_i$ is the $(i,0)$-th cofactor of $D_{L_2}$.
\end{prop}
\begin{proof} The proof follows from similar arguments as in Proposition \ref{P32}.
\end{proof}
At the end of this section, we provide the compositional inverse of the permutation trinomial $f(X):=\alpha X^{q(q^2-q+1)}+\beta X^{q^2-q+1}+2X$, as given in \cite [Theorem 1] {XLXZT} for $\alpha =1$ and $\beta =1$. We observe that for $\alpha=1$ and $\beta=1$, $f(X)=f_5(X^q)$, where $f_5(X)=2X^{q^2}+X^{q^2-q+1}+X^{q^2+q-1}$. Thus, $f(X)$ is a PP if and only if $f_5(X)$ is PP over $\F_{q^3}$, and $f^{-1}(X)=(f_5^{-1}(X))^{q^2}$. In the following theorem, we prove that $f_5(X)$ is a permutation trinomial and determine its compositional inverse.
 \begin{thm}\label{T35}
     Let $q=p^m$, where $m$ is a positive integer and $p$ is an odd prime. Then $f_5(X)=2X^{q^2}+X^{q^2-q+1}+X^{q^2+q-1}$ is a permutation polynomial over $\F_{q^3}$. Moreover, $$f_5^{-1}(X)=\frac{- X+ X^q- X^{q^2}+2 X^{\frac{q^2+1}{2}}}{4}.$$
 \end{thm}
 \begin{proof}
     For proving $f_5(X)$ is a permutation polynomial over $\F_{q^3}$, we shall prove that for each $A\in \F_{q^3}$, the equation
  \begin{equation}\label{me33}
  2X^{q^2}+X^{q^2-q+1}+X^{q^2+q-1}=A
  \end{equation}
     has a unique solution in $\F_{q^3}.$ It is clear that if $A=0$, then $X=0$ is a solution to Equation \eqref{me33}. Now, we will prove that for $A=0$, Equation \eqref{me33} has no solution in $\F_{q^3}^*.$ To prove this, it is enough to show that
     \begin{equation*}
      2X^{q-1}+1+X^{2q-2}=0 
     \end{equation*}
     has no solution in $\F_{q^3}^*$. The above equation is equivalent to $(X^{q-1}+1)^2=0$. It gives that $X^{q-1}=-1$, which implies that $X^{q^3}=-X.$ Thus, we get $2X=0$ or $X=0,$ which is a contradiction.

    We now assume that $A\neq 0$.  It is then clear that $X=0$ cannot be a solution of Equation \eqref{me33}. Let $X\in \F_{q^3}^*$ be a solution of Equation \eqref{me33}. We suppose that $Y=X^q, Z=X^{q^2}$, $B=A^q,$ and  $C=A^{q^2}$. Using Equation \eqref{me33}, we obtain the following system of equations
   \begin{equation}\label{e312}
        2Z+\frac{XZ}{Y}+\frac{YZ}{X}=A,
    \end{equation}
    \begin{equation}\label{e313}
       2X+\frac{YX}{Z}+\frac{ZX}{Y}=B,
          \end{equation}
   \begin{equation}\label{e314}
      2Y+\frac{ZY}{X}+\frac{XY}{Z}=C.
   \end{equation}
       From Equation \eqref{e312}, we have 
       \[
       Z=\frac{AXY}{(X+Y)^2}
       \]
       as $X+Y \neq 0$ for $X \neq 0$.
       By substituting the value of $Z$  into Equations \eqref{e313}, \eqref{e314}, and after some simplifications, we get the following  equations
    \begin{equation*}
    \begin{split}
       l_1(Y):=&Y^4+4XY^3+(2AX+6X^2-AB)Y^2+(4X^3-2ABX+4AX^2)Y+X^4+2AX^3-ABX^2\\&+A^2X^2=0,
    \end{split}
    \end{equation*}
    and
    \begin{equation*}
    \begin{split}
        l_2(Y):=&Y^4+(2A+4X)Y^3+(4AX+A^2+6X^2-AC)Y^2+(2AX^2+4X^3-2ACX)Y+X^4\\&-ACX^2=0.
    \end{split}
    \end{equation*}
    The resultant $R(l_1,l_2,Y)$ of $l_1(Y)$ and $l_2(Y)$ is given by
    \begin{equation*}
        R(l_1,l_2,Y):=A^8X^4(16X^2 + 8XA - 8XB + 8XC + A^2 - 2AB - 2AC + B^2 - 2BC + C^2).
    \end{equation*}
   Since $l_1(Y)$ and $l_2(Y)$ have a common root, $R(l_1,l_2,Y)=0$. As $A\neq 0$ and $X\neq 0$, thus,  
   \begin{equation*}
       16X^2 + 8XA - 8XB + 8XC + A^2 - 2AB - 2AC + B^2 - 2BC + C^2=0,
   \end{equation*}
   which gives that
   \begin{equation*}
       X=\frac{-A+B-C\pm2\sqrt{AC}}{4}.
   \end{equation*}
   We shall show that $ X=\frac{-A+B-C-2\sqrt{AC}}{4}=\frac{-A+A^q-A^{q^2}-2A^{\frac{q^2+1}{2}}}{4}$ is not a solution of Equation \eqref{me33}. If possible, we suppose that $ X=\frac{-A+A^q-A^{q^2}-2A^{\frac{q^2+1}{2}}}{4}$ is a solution of  Equation \eqref{me33}. Then Equation \eqref{me33} is satisfied together with the equation $4X=-A+A^q-A^{q^2}-2A^{\frac{q^2+1}{2}}$.
   We substitute the value of $A$ from Equation \eqref{me33} into the equation $4X=-A+A^q-A^{q^2}-2A^{\frac{q^2+1}{2}}$.  After some simplifications, it gives that 
   \begin{equation}\label{e315}
       4X=2(-X^{q^2}-X^q+X-X^{q^2+q-1})-2P_A(X)^{\frac{1}{2}},
   \end{equation}
   where $P_A(X)=A^{q^2}A=(2X^{q^2}+X^{q^2-q+1}+X^{q^2+q-1})^{q^2+1}$. We will compute $P_A(X)$, separately,
   \begin{equation*}
   \begin{split}
       P_A(X)&=A^{q^2}A\\
       &=(2X^{q^2}+X^{q^2-q+1}+X^{q^2+q-1})^{q^2}(2X^{q^2}+X^{q^2-q+1}+X^{q^2+q-1})\\
       & = 4X^{q^2+q}+2X^{q^2+2q-1}+2X^{q^2+1}+2X^{q+1}+X^{2q}+X^2+2X^{2q^2+q-1}+X^{2(q^2+q-1)}+X^{2q^2},
       \end{split}
   \end{equation*}
   which is equivalent to $$P_A(X)=(X+X^q+X^{q^2}+X^{q^2+q-1})^2.$$ Next, we use $P_A(X)$ in Equation \eqref{e315} to obtain 
   \begin{equation*}
   X+X^q+X^{q^2}+X^{q^2+q-1}=0,
   \end{equation*}
  which implies
\begin{equation*}
1+X^{q-1}+X^{q^2-1}+X^{q^2+q-2}=0.
\end{equation*}  
Let $t=X^{q-1}$, then the above equation implies  $1+t+t^{q+1}+t^{q+2}=0$. It yields that  $t=-1$ or $t^{q+1}+1=0$. If $t=-1$, then $X^{q-1}=-1$, that is, $X^{q^3}=-X$. Thus, we have $2X=0$ or $X=0$, which is a contradiction. Now, if $t^{q+1}=-1$, then we get $X^{q^2-1}=-1$. It would imply that $X=X^{q^3}=-X^q$, therefore, again we obtain $2X=0$ or $X=0,$ a contradiction.  Thus, $$X = \frac{-A+B-C-2\sqrt{AC}}{4}=\frac{-A+A^q-A^{q^2}-2A^{\frac{q^2+1}{2}}}{4}$$
is not a solution of Equation \eqref{me33}, and $$X=\frac{-A+B-C+2\sqrt{AC}}{4}=\frac{-A+A^q-A^{q^2}+2A^{\frac{q^2+1}{2}}}{4}$$ is the unique solution of  Equation \eqref{me33}. Hence, $f_5(X)$ is a permutation polynomial over $\F_{q^3}$.

For the compositional inverse,  let $X \in \F_{q^3}^{*}$, $g_1(X)=X, g_2(X)=X^q, g_3(X)=X^{q^2}, h_1(X)=g_1\circ f_1(X)=f_1(X), h_2(X)=g_2\circ f_1(X)=f_1^{q}(X)$, and $h_3(X)=g_3\circ f_1(X)=f_1^{q^2}(X)$. Then from the above discussion, 
\[
X=\frac{-A+B-C+2\sqrt{AC}}{4},
\]
where $A=h_1(X), B=A^q=h_2(X)$, and $C=A^{q^2}=h_3(X)$.

Thus, in view of  Lemma \ref{L22}, we get $F(X_1, X_2, X_3)=\frac{-X_1+X_2-X_3+2\sqrt{X_1X_3}}{4}$ and the compositional inverse of $f_5(X)$ is given by 
\begin{equation*}
f_5^{-1}(X)=F(X, X^q, X^{q^2})=\frac{-X+X^q-X^{q^2}+2X^{\frac{q^2+1}{2}}}{4}.
\end{equation*}
\end{proof}
\begin{rmk}
The compositional inverse $\frac{-X+X^q-X^{q^2}+2X^{\frac{q^2+1}{2}}}{4}$ of $f_5(X)$ is a permutation quadrinomial of the form $aX^d+L(X)$ over $\F_{q^3}$. Here, the interesting fact is that Theorem \ref{T35} gives a new permutation quadrinomial over $\F_{q^3}$. Moreover, in the next section,  we study the inequivalence of this permutation quadrinomial with the known ones.
\end{rmk}

\section{QM-Inequivalence of permutation polynomials}\label{S4}
This section studies the QM-inequivalence of permutation polynomials proposed in Section \ref{S3} with the known classes of permutation polynomials and among themselves. Throughout this section, $\mathbb{Z}_{q^3-1}$ denotes the ring of integers modulo $q^3-1$.

\begin{def1}\textup{\cite{WYDM}} Two permutation polynomials $f(X)$ and $g(X)$ in $\F_{q}[X]$ are quasi-multiplicative (QM) equivalent if there exist an integer $1\leq d \leq q-2$ such that $\gcd(d, q-1)=1$ and $f(X)=ag(bX^d)$ for some $a,b \in \F_q^*.$
\end{def1}

\begin{prop}\label{P1}
   The PPs of Table~\ref{T1} are QM-inequivalent among themselves. 
\end{prop}
\begin{proof}
    We first show that the permutation polynomial $f_{1}(X)$ is QM-inequivalent to the permutation polynomial $f_{2}(X)$ listed in Table \ref{T1}. On the contrary, we assume that $f_1(X)$ is QM-equivalent to $f_{2}(X)$. This implies that there exists an integer $1\leq d \leq q^3-2$ such that $\gcd(d, q^3-1)=1$ and $f_{2}(X)=af_{1}(bX^d)$ for some $a,b \in \F_{q^3}^{*}$.  This would turn out that the sets of exponents of the polynomials  $af_{1}(bX^d)$ and $f_2(X)$ are the same in $\mathbb{Z}_{q^3-1}$, that is, 
 \[
        A:=\{d2q^2, d(q^2-q+1), d(q^2+q-1)\}=\{1+q, 1+q^2, 2q^2+2q\}=:B.
 \]
Suppose that $d2q^2 \equiv 1+q \pmod {q^3-1}$, which gives $d \equiv \frac{q^2+q}{2} \pmod {q^3-1}$. Then $d(q^2-q+1) \equiv q \pmod {q^3-1}$, however, any exponent from $B$ is not equivalent to $q $ modulo $q^3-1$. Thus, $d2q^2 \not \equiv 1+q \pmod {q^3-1}$. Now, if $d2q^2 \equiv 1+q^2 \pmod {q^3-1}$, then $d \equiv \frac{q^3+q}{2} \pmod {q^3-1}$ implies that $d(q^2-q+1) \equiv \frac{q^5-q^4+2q^3-q^2+q}{2} \equiv 1 \pmod {q^3-1}$, which is not possible. Finally, we assume that $d2q^2 \equiv 2q^2+2q \pmod {q^3-1}$.  Then $d(q^2-q+1) \equiv 2q^2 \pmod {q^3-1}$, a contradiction as $2q^2 \not \equiv t \pmod {q^3-1}$ for any $t \in B$. It concludes that there is no such $d$ exists and  $f_1(X)$ is QM-inequivalent to $f_2(X)$. Similarly, we can prove that $f_1(X)$ is QM-inequivalent to $f_3(X)$ and $f_4(X)$.

We now study the QM-equivalence between $f_2(X)$ and $f_t(X)$; $3 \leq t \leq 4$. For $t=3$, if possible, suppose that there exists an integer $1\leq d \leq q^3-2$ such that $\gcd(d, q^3-1)=1$ and $f_{3}(X)=af_{2}(bX^d)$ for some $a,b \in \F_{q^3}^{*}$. Then we have 
 \[
        \{d(q+1), d(q^2+1), d(2q^2+2q)\}=\{1+q, 1+q^2, 2q+2\} 
  \]
  in $\mathbb{Z}_{q^3-1}$.
Clearly, $d(q+1) \equiv q+1 \pmod {q^3-1}$ is not possible. Now, if $d(q+1) \equiv 1+q^2 \pmod {q^3-1}$, then $d \equiv \frac{q^2+1}{q+1} \pmod {q^3-1}$, i.e., $d \equiv q^2 \pmod {q^3-1}$ . However, clearly $d(q^2+1)\not \equiv 1+q \pmod {q^3-1}$ and for large $q$, $d(q^2+1)\not \equiv 2q+2 \pmod {q^3-1}$. For $d(q+1) \equiv 2q+2 \pmod {q^3-1}$, $d \equiv 2 \pmod {q^3-1}$. Again in this case, $d(q^2+1)\not \equiv 1+q \pmod {q^3-1}$ and $d(q^2+1)\not \equiv 1+q^2 \pmod {q^3-1}$. Thus, $f_2(X)$ is QM-inequivalent to $f_3(X)$. 
Using similar technique, one can prove that  $f_2(X)$ and $f_3(X)$ are QM-inequivalent to $f_4(X)$. Hence, $f_t(X)$'s, $1 \leq t \leq 4$, are QM-inequivalent to each other.  
\end{proof}
\begin{prop}\label{P2}
Permutation trinomial $f_1(X)$ listed in Table \ref{T1} is QM-inequivalent to all the known permutation trinomials listed in Table \ref{T2}. 
\end{prop}
\begin{proof}
 First we will show that $f_1(X)$ is QM-inequivalent to $k_1(X)$. Suppose, for the sake of contradiction, that $f_1(X)$ is QM-equivalent to $k_1(X)$. Then there exists a positive integer $1 \leq d \leq q^3-2$ such that $\gcd(d,q^3-1)=1$ and $k_1(X)=af_1(bX^d)$ for some $a,b \in \F_{q^3}^{*}$. This implies that the following sets of exponents of $k_1(X)$ and $af_1(bX^d)$ are same in $\mathbb{Z}_{q^3-1}$, i.e., 
\[
\{d2q^2,d(q^2-q+1),d(q^2+q-1)\}=\{2,q+1,q^2+1\}.
\]
If possible, suppose that $d2q^2 \equiv 2 \pmod {q^3-1}$, then $d \equiv q \pmod {q^3-1}$, which gives $d(q^2+q-1)\equiv q^2-q+1  \pmod {q^3-1}$.  However, for large $q$, $q^2-q+1 \notin \{q+1,q^2+1\}$ in modulo $q^3-1$.  Next, if $d2q^2 \equiv q+1 \pmod {q^3-1}$, then in this case $d \equiv \frac{q^2+q}{2}  \pmod {q^3-1}$. It implies $d(q^2+q-1) \equiv \frac{q^4+2q^3-q}{2} \equiv 1 \pmod {q^3-1}$ but  $ 1 \not \equiv 2, q^2+1 \pmod {q^3-1}$. Now, we assume that $d2q^2 \equiv q^2+1 \pmod {q^3-1}$, which gives $ d \equiv \frac{q^3+q}{2} \pmod {q^3-1}$. In this case, $d(q^2-q+1) \equiv 1 \not \equiv 2,q+1 \pmod {q^3-1}$. Therefore, $f_1(X)$ is QM-inequivalent to $k_1(X)$. 

Next, we will prove that $f_1(X)$ is QM-inequivalent to $k_2(X)$. On the contrary, we assume that $f_1(X)$ is QM-equivalent to $k_2(X)$. Thus, $k_2(X)=af_1(bX^d)$ for some $1 \leq d \leq q^3-2$ with $\gcd(d,q^3-1)=1$ and $a,b \in \F_{q^3}^{*}$. Again, we have 
\[
\{d2q^2,d(q^2-q+1),d(q^2+q-1)\}=\{1,q^2-q+1,q^2+q-1\}
\]
in $\mathbb{Z}_{q^3-1}$.  Assume that $d2q^2 \equiv 1 \pmod {q^3-1}$, then $d \equiv \frac{q}{2} \pmod {q^3-1}$ implies that $d(q^2-q+1) \equiv \frac{q^3-q^2+q}{2} \pmod {q^3-1}$. Moreover, either $\frac{q^3-q^2+q}{2} \equiv q^2-q+1 \pmod {q^3-1}$ or $\frac{q^3-q^2+q}{2} \equiv q^2+q-1 \pmod {q^3-1}$, which gives that $\frac{q}{2} \equiv 1 \pmod {q^3-1}$ or $\frac{q^2}{2} \equiv 1 \pmod {q^3-1}$. This is not true for $q>2$. Now consider $d2q^2 \equiv q^2-q+1 \pmod {q^3-1}$. Thus, $d \equiv \frac{q^3-q^2+q}{2}  \pmod {q^3-1}$ and it gives that  $d(q^2-q+1) \equiv \frac{q}{2} (q^2-q+1)^2 \pmod {q^3-1}$. There are two possibilities either $\frac{q}{2} (q^2-q+1)^2 \equiv 1 \pmod {q^3-1}$ or $\frac{q}{2} (q^2-q+1)^2 \equiv q^2+q-1 \pmod {q^3-1}$. First condition gives that $q^2+q-1 \equiv 0 \pmod {q^3-1}$ and second one gives that $3q^2+3q-5 \equiv 0 \pmod {q^3-1}$, which are not possible. Using similar techniques, we can deal with the case when $d2q^2 \equiv q^2+q-1 \pmod {q^3-1}$. Thus, $f_1(X)$ is QM-inequivalent to $k_2(X)$. Similarly, QM-inequivalence of $f_1(X)$ with $k_3(X)$ and $k_4(X)$ can be proved. 
 
 Now, we compare $f_1(X)$ with $k_5(X)$ and show that  $f_1(X)$ is QM-inequivalent to $k_5(X)$. We know that if $f_1(X)$ is QM-equivalent to $k_5(X)$, then 
 \[
\{d2q^2,d(q^2-q+1),d(q^2+q-1)\}=\{1,q^2,q^2-q+1\}
\]
in $\mathbb{Z}_{q^3-1}$ for some $1 \leq d \leq q^3-2$ satisfying $\gcd(d,q^3-1)=1$. If possible, we assume that $d2q^2 \equiv 1 \pmod {q^3-1}$, then $d \equiv \frac{q}{2} \pmod {q^3-1}$ and it gives $d(q^2-q+1) \equiv \frac{q^3-q^2+q}{2} \pmod {q^3-1}$. In this case, either $\frac{q^3-q^2+q}{2} \equiv q^2 \pmod {q^3-1}$ or $\frac{q^3-q^2+q}{2} \equiv q^2-q+1$, which implies $q^2-3q+1 \equiv 0 \pmod {q^3-1}$ or $q \equiv 2 \pmod {q^3-1}$ as $\gcd(q^2-q+1,q^3-1)=1$. These two cases are not possible for $q>2$. In a similar way, we can exclude the case $d2q^2 \equiv q^2 \pmod {q^3-1}$. Next, we assume that $d2q^2 \equiv q^2-q+1 \pmod {q^3-1}$, which renders that $d \equiv \frac{q^3-q^2+q}{2} \pmod {q^3-1}$. Thus, either $\frac{q^3-q^2+q}{2}(q^2-q+1) \equiv 1 \pmod {q^3-1}$ or $\frac{q^3-q^2+q}{2}(q^2-q+1) \equiv q^2 \pmod {q^3-1}$. The condition $\frac{q^3-q^2+q}{2}(q^2-q+1) \equiv 1 \pmod {q^3-1}$ gives that $q^2+q-1 \equiv 0 \pmod {q^3-1}$, which is a contradiction. The second condition implies that $3q^2+q-3 \equiv 0 \pmod {q^3-1}$, which is again not possible for large $q$. Hence, $f_1(X)$ is QM-inequivalent to $k_5(X)$. By using a similar approach, one can prove that $f_1(X)$ is QM-inequivalent to $k_i(X)$ for $i \geq 6$.

\end{proof}

\begin{prop}\label{P3}
Permutation trinomials $f_2(X)$, $f_3(X)$ and $f_4(X)$ listed in Table \ref{T1} are QM-inequivalent to all the known permutation trinomials listed in Table \ref{T2}. 
\end{prop}
\begin{proof}
Using similar techniques as adopted in Proposition \ref{P2}, we can prove that  $f_2(X)$, $f_3(X)$, and $f_4(X)$  are QM-inequivalent to all the known permutation trinomials listed in Table \ref{T2}. 
\end{proof}
\begin{prop}\label{P4}
The compositional inverse $f_5^{-1}(X)=\frac{- X+ X^q- X^{q^2}+2 X^{\frac{q^2+1}{2}}}{4}$ obtained in Theorem \ref{T35} is  QM-inequivalent to all the known permutation quadrinomials listed in Table \ref{T3}. 
\end{prop}
\begin{proof}  
 We first prove that $f_5^{-1}(X)$ is QM-inequivalent to $Q_1(X)$.  On the contrary, we assume that  $f_5^{-1}(X)$ is QM-equivalent to $Q_1(X)$. Thus, there exists a positive integer $1 \leq d \leq q^3-2$ with $\gcd(d,q^3-1)=1$ such that $Q_1(X)=af_5^{-1}(bX^d)$ for some $a,b \in \F_{q^3}^{*}$. This implies that the sets of exponents of $Q_1(X)$ and $af_5^{-1}(bX^d)$ are same in $\mathbb{Z}_{q^3-1}$, i.e.,
\[
A:=\{d,dq,dq^2,d\frac{q^2+1}{2}\}=\{1,q,q^2,q^2+q-1\}=:B.
\]
It is clear that $d \equiv 1 \pmod {q^3-1}$ is not possible. Next, we assume that $d \equiv q \pmod {q^3-1}$. Then $d\frac{q^2+1}{2} \equiv \frac{q^3+q}{2} \pmod {q^3-1}$. Therefore, either $\frac{q^3+q}{2} \equiv 1 \pmod {q^3-1}$ or $\frac{q^3+q}{2} \equiv q^2 \pmod {q^3-1}$ or $\frac{q^3+q}{2} \equiv q^2+q-1 \pmod {q^3-1}$, which gives $q \equiv 1 \pmod {q^3-1}$  or $2q^2-q-1 \equiv 0 \pmod {q^3-1}$ or $2q^2+q-3 \equiv 0 \pmod {q^3-1}$. However,  for large $q$, all these cases are not possible. Similarly, one can deal with the case when $d \equiv q^2 \pmod {q^3-1}$. Now, if  $d \equiv q^2+q-1 \pmod {q^3-1}$, then $A=\{q^2+q-1,q^2-q+1,q^3-q^2+q,q\}$. Clearly, no element from $A$ is congruent to $1$ modulo $q^3-1$. Hence, $f_5^{-1}(X)$ is QM-inequivalent to $Q_1(X)$. Next, we compare $f_5^{-1}(X)$ with $Q_2(X)$ and show that they are  QM-inequivalent. If possible, suppose that there exists an integer $1 \leq d \leq q^3-2$ with $\gcd(d,q^3-1)=1$ such that $Q_2(X)=af_5^{-1}(bX^d)$ for some $a,b \in \F_{q^3}^{*}$. This implies $\{d,dq,dq^2,d\frac{q^2+1}{2}\}=\{1,q^2,q^2-q+1,q^2+q-1\}$ in $\mathbb{Z}_{q^3-1}$. If $d \equiv 1 \pmod {q^3-1}$ or $d \equiv q^2 \pmod {q^3-1}$, then 
$$\left\{d,dq,dq^2,d\frac{q^2+1}{2}\right\}=\left\{1,q,q^2, \frac{q^2+1}{2}\right\} \text{ or } \left\{d,dq,dq^2,d\frac{q^2+1}{2}\right\}=\left\{1,q,q^2, \frac{q^3+q}{2}\right\}$$ in $\mathbb{Z}_{q^3-1}$.
 Therefore, $q \equiv q^2-q+1 \pmod {q^3-1}$ or $q \equiv q^2+q-1 \pmod {q^3-1}$, which is clearly not true. Moreover, if $d \equiv q^2-q+1$ or $d \equiv q^2+q-1$, then  $q^3-q^2+q \equiv 1 \pmod {q^3-1}$ or $q^3-q^2+q \equiv q^2 \pmod {q^3-1}$.  This is also not possible. Thus,  $f_5^{-1}(X)$ is QM-inequivalent to $Q_2(X)$. Using the similar approach, we can prove that  $f_5^{-1}(X)$ is QM-inequivalent to $Q_i(X)$ for $i \geq 3$.
\end{proof}

\begin{longtable}{|m{1em}| m{11em}| m{11em}| 
m{10em}|}
\caption{Permutation trinomials over $\F_{q^3}$ obtained in section 3.}\label{T1}
\endlastfoot
\hline 
$t$ & $ f_t(X)$ &   Conditions on $m$  & 
      References\\
\hline
1 & $X^{2q^2}+aX^{q^2-q+1}+aX^{q^2+q-1}$; $a\in \F_{q}^*$ & all $m$ & Theorem 3.1\\
\hline
2 & $aX^{1+q}+aX^{1+q^2}+X^{2q^2+2q}$; $a\in \F_{q}^*$ & all $m$ & Theorem 3.2\\
\hline
3 & $X^{1+q}+X^{1+q^2}+X^{2q+2}$ & $m\not\equiv 1 \pmod 3$ & Theorem 3.3\\
\hline
4 & $X^{1+q}+X^{1+q^2}+X^{2q^2+2}$ & $m\not\equiv 2 \pmod 3$ & Theorem 3.4\\
\hline
\end{longtable}
\begin{longtable}{|m{1em}| m{11em}| m{11em}| 
m{10em}|}
\caption{Known classes of permutation trinomials over $\F_{q^3}$.}\label{T2}
\endlastfoot
\hline 
  $i$ & $ k_i(X)$ &   Conditions on $m$  & 
      References\\ 
\hline
1 & $\gamma X^2+(\gamma X)^{q+1}+X^{q^2+1}$; & all $m$ & \cite[Theorem 4]{BCHO}\\
& $\gamma^{q^2+q+1}\neq 1$ & &\\
\hline
2 & $X+AX^{q^2-q+1}+X^{q^2+q-1}$; & $m\not\equiv 2 \pmod 3$ & \cite[Theorem 3.2]{GGS}\\
& $A\in \F_{q}$ with $A^3=1$ & & \\
\hline
3 & $X+AX^{q^3-q^2+q}+X^{q^2+q-1}$; & $m\not\equiv 1 \pmod 3$ & \cite[Theorem 3.4]{GGS}\\
& $A\in \F_{q}$ with $A^3=1$ & & \\
\hline
4 & $X+X^{q^2}+X^{q^3-q^2+q}$ &  $m\not\equiv 1 \pmod 3$ & \cite[Theorem 4.4]{LQCL}\\
\hline
5 & $bX+aX^{q^2}+X^{q^2-q+1}$; & all $m$ & \cite[Theorem 1]{PXLZ}\\
& $a^{q^2+q+1}=b^{q^2+q+1}$, $a^qb\in \F_{q}^*\setminus\{1\}$ & & \\
\hline
6 & $aX+a^{q+2}X^q+X^{q^2-q+1}$; & $m\not\equiv 2 \pmod 3$ and & \cite[Theorem 2]{PXLZ}\\
& $a^{2(q^2+q+1)}+a^{q^2+q+1}=1$,  & m is even & \\
\hline
7 & $a^{q^2+2}X^q+aX^{q^2}+X^{q^2-q+1}$; & $m\not\equiv 1 \pmod 3$ and & \cite[Theorem 3]{PXLZ}\\
& $a^{2(q^2+q+1)}+a^{q^2+q+1}=1$,  & m is even & \\
\hline
 & $X+aX^{q^2-q+1}+bX^{q(q^2-q+1)}$; & $m\equiv 0 \pmod 3$ & \cite[Theorem 4]{PXLZ}\\ 
& $a^{q^2+q+1}=1$, $ab\in \F_{2^3}^*$, $\Tr_{1}^{3}(ab)=0$ & & \\
8 & $X+aX^{q^2-q+1}+bX^{q(q^2-q+1)}$; & $m\equiv 0 \pmod 3$ & \cite[Theorem 4]{PXLZ}\\ 
& $a=\beta c$,  $b=c^{q^2+q}$ with $c^{q^2+q+1}=1,$ $\beta\in \F_{2^3}^*$, $\Tr_{1}^{3}(\beta)=0$ & & \\
 & $X+aX^{q^2-q+1}+bX^{q(q^2-q+1)}$; & $m$ is even & \cite[Theorem 4]{PXLZ}\\ 
& $a^{2(q^2+q+1)}+a^{q^2+q+1}=1$ and $ab=1$ & & \\
\hline
9 & $X+aX^{q(q^2-q+1)}+bX^{q^2(q^2-q+1)}$; & $m\equiv 3 \pmod 9$ & \cite[Theorem 5]{PXLZ}\\ 
& $a=\beta c$, $b=\beta^2c^{2q+1}$, $\beta\in \F_{2^3}^*\setminus\{1\}$, $\beta^3+\beta^2+1=0$ and $c^{q^2+q+1}=1$ & & \\
\hline
10 & $X+aX^{q^2(q^2-q+1)}+bX^{q^2-q+1}$; & $m\equiv 3 \pmod 9$ & \cite[Theorem 6]{PXLZ}\\
& $a=\beta^4 c^{2q^2+1}$, $b=\beta c$, $\beta\in \F_{2^3}^*$, $\beta^3+\beta+1=0$ and $c^{q^2+q+1}=1$ & & \\
\hline
11 & $vX+X^{q+1}+X^{q^2+1}$; & all $m$ & \cite[Theorem 4]{TZH}\\
& $v\in\F_{q}\setminus\{0,1\}$ & & \\
\hline
12 & $X+X^{q^2}+X^{q^2+q-1}$ & $m\not\equiv 2 \pmod 3$ & \cite[Theorem 2.2]{WZZ}\\ 
\hline
13 & $X^{q^2+q-1}+aX+bX^{q^2-q+1}$ & all $m$ & \cite[Corollary 3.3]{ZKZPL}\\ 
\hline
14 & $X^{q^2+q-1}+aX+bX^{q^3-q^2+q}$ & all $m$ & \cite[Corollary 3.7]{ZKZPL}\\ 
\hline
\end{longtable}
\begin{longtable}{|m{1em}| m{15em}| m{11em}| 
m{10em}|}
\caption{Known classes of permutation quadrinomials over $\F_{q^3}$, for odd $q$.}\label{T3}
\endlastfoot
\hline 
$i$ & $ Q_i(X)$ &   Conditions  & 
      References
      \\
\hline
1 & $X^{q^2+q-1}+AX^{q^2}+BX^{q}+CX$ & see \cite[Theorem 3.6]{WZBW}& \cite[Theorem 3.6]{WZBW}\\
\hline
2 & $X^{q^2+q-1}+AX^{q^2-q+1}+BX^{q^2}+CX$ & see \cite[Theorem 3.7]{WZBW} & \cite[Theorem 3.7]{WZBW}\\
\hline
3 & $X^{q^2+q-1}+AX^{q^3-q^2+q}+BX^{q}+CX$ & see \cite[Theorem 3.8]{WZBW} &  \cite[Theorem 3.8]{WZBW}\\
\hline
4 & $X^{q^2+q-1}+AX^{q^2-q+1}+BX^{q}+CX$ & see \cite[Theorem 3.9]{WZBW} &  \cite[Theorem 3.9]{WZBW} \\
\hline
      5 & $X\pm X^q+X^{q^2}+X^{q^2-q+1}$ &  &  \cite[Theorems 3.3, 3.4]{CKPW} \\
\hline
\end{longtable}

\end{document}